\documentclass[a4paper,11pt]{amsart}
\usepackage[arrow,curve,matrix]{xy}
\usepackage[a4paper,twoside,centering]{geometry}

\title{On Jacobian algebras from closed surfaces}
\author{Sefi Ladkani}
\address{%
Mathematical Institute of the University of Bonn \\
Endenicher Allee 60 \\
53115 Bonn, Germany}
\urladdr{http://www.math.uni-bonn.de/people/sefil}
\email{sefil@math.uni-bonn.de}

\thanks{The author is supported by DFG grant LA 2732/1-1 in the
framework of the priority program SPP 1388 ``Representation
theory''.}

\DeclareMathOperator{\Hom}{Hom}
\DeclareMathOperator{\PSL}{PSL}
\DeclareMathOperator{\rank}{rank}

\newcommand{\balpha}{\bar{\alpha}}
\newcommand{\bbeta}{\bar{\beta}}
\newcommand{\cC}{\mathcal{C}}
\newcommand{\gl}{\lambda}
\newcommand{\gL}{\Lambda}
\newcommand{\vphi}{\varphi}
\newcommand{\cP}{\mathcal{P}}
\newcommand{\bZ}{\mathbb{Z}}
\newcommand{\half}{{^1}{\!\!}/{\!}{_2}}

\theoremstyle{plain}
\newtheorem*{theorem*}{Theorem}
\newtheorem{theorem}{Theorem}[section]
\newtheorem{prop}[theorem]{Proposition}
\newtheorem{lemma}[theorem]{Lemma}
\newtheorem{cor}[theorem]{Corollary}
\newtheorem*{cor*}{Corollary}

\theoremstyle{definition}
\newtheorem{defn}[theorem]{Definition}
\newtheorem{remark}[theorem]{Remark}

\numberwithin{equation}{section}

\begin{document}

\begin{abstract}
We show that the quivers with potentials associated to ideal
triangulations of marked surfaces with empty boundary are not rigid,
and their completed Jacobian algebras are finite-dimensional and
symmetric.
\end{abstract}

\maketitle

\section{Introduction}

In~\cite{Labardini09} Labardini-Fragoso associated a quiver with
potential to any ideal triangulation of a surface with marked points
in such a way that flips of triangulations correspond to mutations
of the associated quivers with potentials, thus providing a link
between the work of Fomin, Shapiro and Thurston~\cite{FST08} on
cluster algebras arising from marked surfaces and the theory of
quivers with potentials initiated by Derksen, Weyman and
Zelevinsky~\cite{DWZ08}.

When the surface has non-empty boundary, the potential associated to
any ideal triangulation is rigid and its Jacobian algebra is
finite-dimensional~\cite{Labardini09}. However, when the surface has
empty boundary, it was conjectured that the potential associated to
any ideal triangulation is not
rigid~\cite[Conjecture~34]{Labardini09}. The question whether its
Jacobian algebra is finite-dimensional or not has been open for some
time, see~\cite[Problem~8.1]{Labardini09b},
\cite[Question~6.4]{Labadrini12} and the
survey~\cite[Remark~3.17]{Amiot11}. The only cases where
finite-dimensionality has been established so far are the
once-punctured torus \cite[Example~8.2]{Labardini09b} and recently
the spheres with arbitrary number of
punctures~\cite{TrepodeValdivieso12}.

Our main result, stated in the following theorem, completely settles
these questions. Recall that the auxiliary algebraic data needed to
define the potential consists of a non-zero scalar (from a fixed
field) for each puncture.

\begin{theorem*}
Let $(S,M)$ be a surface with marked points and empty boundary.
\begin{enumerate}
\renewcommand{\theenumi}{\alph{enumi}}
\item
If $(S,M)$ is not a sphere with $4$ punctures, then for any choice
of scalars the quiver with potential associated to any ideal
triangulation of $(S,M)$ is not rigid and its (completed) Jacobian
algebra is finite-dimensional and symmetric.

\item
If $(S,M)$ is a sphere with $4$ punctures, then the same conclusion
holds provided that the product of the scalars is not equal to $1$.
\end{enumerate}
\end{theorem*}

The theorem provides in particular an explicit construction of
infinitely many families of symmetric, finite-dimensional Jacobian
algebras.

\medskip

As a consequence of the theorem we can associate a $\Hom$-finite
cluster category to any marked surface with empty boundary  in a
similar way as in the case of non-empty
boundary~\cite[\S3.4]{Amiot11}. It is the generalized cluster
category of Amiot~\cite{Amiot09} associated to the Jacobian algebra
corresponding to (any) ideal triangulation. In the case of a sphere
with $4$ punctures, this category is a tubular cluster category
studied by Barot and Geiss~\cite{BarotGeiss12}.

\begin{cor*}
Let $(S,M)$ be a surface with marked points and empty boundary. Then
there is a $\Hom$-finite triangulated $2$-Calabi-Yau category
$\cC_{(S,M)}$ with a cluster-tilted object for each ideal
triangulation.
\end{cor*}

We outline our strategy for proving the theorem. Since the
properties of non-rigidity and finite-dimensionality of Jacobian
algebras are preserved under mutations of quivers with
potentials~\cite{DWZ08} and any two ideal triangulations of a
surface with marked points can be connected by a sequence of flips,
it suffices to consider only one triangulation. Therefore we can
avoid technical complications by dealing only with those
triangulations which are suitably ``nice''.

In Section~\ref{sec:model} we consider triangulations with at least
three arcs incident to every puncture and develop a combinatorial
model for the associated quiver with potential. In
Section~\ref{sec:star} we introduce the additional
conditions~\eqref{e:star} and~\eqref{e:diamond} on the quiver, and
we express them in terms of combinatorial properties of the
corresponding triangulation.

Then, in Section~\ref{sec:rels} we investigate the relations in the
Jacobian algebra of a quiver with potential within the framework of
our model. Some relations always hold, whereas additional relations
are obtained by assuming additional hypotheses involving either
condition~\eqref{e:star} or~\eqref{e:diamond}. Under these
hypotheses we carry out the actual proof in
Section~\ref{sec:algebra}, where we show that the potential is not
rigid (Proposition~\ref{p:nonrigid}) and the Jacobian algebra is
finite-dimensional (Proposition~\ref{p:findim}) and symmetric
(Proposition~\ref{p:symmetric}).

The existence of ``nice'' triangulations is shown in
Section~\ref{sec:exist}. It implies that for any surface with marked
points and empty boundary there is a triangulation whose associated
quiver with potential satisfies one of the conditions~\eqref{e:star}
or~\eqref{e:diamond}, thus allowing us to conclude the proof.

In addition we also compute the Cartan matrices and the centers of
the Jacobian algebras of the quivers with potentials considered in
Section~\ref{sec:algebra}. For the precise statements see
Proposition~\ref{p:Cartan}, Corollary~\ref{c:Cartan} and
Proposition~\ref{p:center}. In particular, the rank of the Cartan
matrix is bounded by the number of punctures, its determinant always
vanishes and the center is the quotient of a polynomial ring (with
as many variables as the arcs in the triangulation) by the ideal
generated by all monomials of degree $2$.

Since the property of a finite-dimensional algebra being symmetric,
as well as its center and the rank of its Cartan matrix are all
invariant under derived equivalence, the extension of the above
results to all the quivers with potentials arising from
triangulations of marked surfaces with empty boundary is now a
consequence of the following result:

\smallskip

\emph{All the Jacobian algebras associated to the ideal
triangulations of a given surface with marked points and empty
boundary are derived equivalent.}

\smallskip

We defer the proof of this result to a subsequent paper dealing with
(weakly) symmetric Jacobian algebras in a broader framework.

\section{Combinatorial model for the quiver with potential}
\label{sec:model}

Let $(S,M)$ be a closed surface with marked points. Recall that $S$
is a compact, connected, oriented Riemann surface with empty
boundary and $M$ is a finite set of points in $S$, called also
punctures. In this section we consider a fixed ideal triangulation
$T$ of $(S,M)$ with the property that
\begin{equation} \tag{T3} \label{e:T3}
\text{at each puncture $p \in M$ there are at least three arcs of
$T$ incident to $p$}
\end{equation}
(where an arc starting and ending at the same puncture is counted
twice). In particular, such a triangulation $T$ does not contain any
self-folded triangles. As we shall see in Section~\ref{sec:exist},
any marked closed surface has such a triangulation.

\subsection{The quiver}

Let $Q$ be the adjacency quiver of $T$ as defined by Fomin, Shapiro
and Thurston~\cite{FST08}. Recall that $Q$ is constructed in the
following way: its vertices are the arcs of $T$, and we add an arrow
from the arc $i$ to the arc $j$ if they are incident to a common
puncture $p$ and the arc $j$ immediately follows $i$ in the
counterclockwise order around $p$.
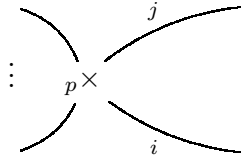
\[
\xymatrix@=1pc{
& & & & \\
{\vdots} & {_p\times} \ar@{-}@/_/[drrr]_i \ar@{-}@/^/[urrr]^j
\ar@{-}@/_/[ul] \ar@{-}@/^/[dl] \\
& & & &
}
\]

\begin{remark} \label{rem:2cycles}
\emph{A-priori}, this process may create $2$-cycles that then have
to be removed when forming the adjacency quiver. However, due to our
assumption~\eqref{e:T3}, this never happens.
\end{remark}

The next proposition lists some basic properties of the quiver $Q$
which will be crucial in our considerations. Denote by $Q_0$ the set
of vertices of $Q$ and by $Q_1$ the set of its arrows.

\begin{prop} \label{p:quiver}
Let $Q$ be the adjacency quiver of the triangulation $T$
satisfying~\eqref{e:T3}. Then:
\begin{enumerate}
\renewcommand{\theenumi}{\alph{enumi}}
\item \label{it:no2cycle}
$Q$ is connected, and there are no loops or $2$-cycles in $Q$.

\item \label{it:degs}
For any $i \in Q_0$, there are exactly two arrows in $Q_1$ starting at $i$ and
two arrows ending at $i$.

\item \label{it:fg_def}
There are invertible maps $f, g : Q_1 \to Q_1$ with the following properties:
\begin{itemize}
\item
For any $\alpha \in Q_1$, the set $\{f(\alpha), g(\alpha)\}$ consists of the
two arrows that start at the vertex which $\alpha$ ends at;
\item
$f^3$ is the identity on $Q_1$.
\end{itemize}
\end{enumerate}
\end{prop}
\begin{proof}
Part~\eqref{it:no2cycle} is evident from the construction.

To show~\eqref{it:degs}, observe that any arc $i$ is a side of
exactly two triangles of $T$, and each such triangle contributes one
arrow starting at $i$ and another ending at $i$.

For part~\eqref{it:fg_def}, we define the maps $f$ and $g$ as
follows. An arrow $\alpha$ corresponds to a pair $i$, $j$ of
consecutive arcs around a common puncture $p$, as in
Figure~\ref{fig:maps}, so that $\alpha$ starts at $i$ and ends at
$j$.

\begin{figure}
\begin{align*}
\begin{array}{c}
\xymatrix@=1pc{
& & & & {\times_q} \ar@{-}@/^/[dd]^k \\
{\vdots} & {_p\times} \ar@{-}@/_/[drrr]_i \ar@{-}@/^/[urrr]^j
\ar@{-}@/_/[ul]_{\ell} \ar@{-}@/^/[dl]^{\ell'} \\
& & & & {\times}
}
\end{array}
& &
\begin{array}{c}
\xymatrix@=0.8pc{
{\bullet_\ell} \\
&& {\bullet_j} \ar[ull]^{g(\alpha)} \ar[drr]^{f(\alpha)} \\
&& && {\bullet_k} \\
&& {\bullet_i} \ar[uu]^{\alpha}
}
\end{array}
\end{align*}
\caption{Definition of the maps $f$ and $g$ on the set of arrows.}
\label{fig:maps}
\end{figure}
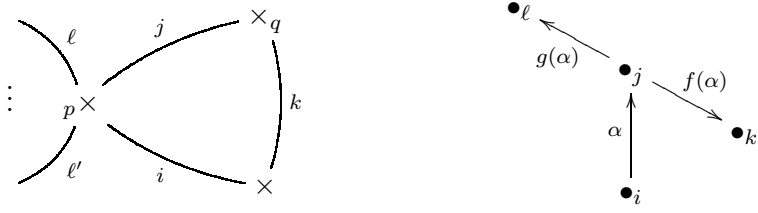

Let $\ell$ be the arc next to $j$ in the counterclockwise order
around $p$. We define $g(\alpha)$ to be the corresponding arrow $j
\to \ell$.

Let $q$ be the puncture at the other end of $j$ and let $k$ be the
arc next to $j$ in the counterclockwise order around $q$. We define
$f(\alpha)$ to be the corresponding arrow $j \to k$. Observe that
the arcs $i$, $j$, $k$ enclose a triangle of $T$, hence
$f^2(\alpha)$ is an arrow $k \to i$ and $f^3(\alpha)$ is the arrow
$\alpha$. In particular, the map $f$ is invertible.

We note that the puncture $q$ may coincide with the puncture $p$ so
that the arc $\ell$ may coincide with $k$, for example in a
triangulation of a once punctured torus. In this case both arrows
$f(\alpha)$ and $g(\alpha)$ start at $j$ and end at $k=\ell$.

Finally, the map $g$ is invertible; indeed, if $\ell'$ is the arc
immediately preceding $i$ in the counterclockwise order around $p$,
then by applying $g$ on the corresponding arrow $\ell' \to i$ we get
$\alpha$.
\end{proof}

Since the map $g$ is invertible, it induces a partition of the
arrows in $Q_1$ into \emph{$g$-orbits}, where the $g$-orbit of an
arrow $\alpha \in Q_1$ is by definition the set of all arrows of the
form $g^i(\alpha)$ for some $i \in \bZ$. Let $n_\alpha$ be the size
of the $g$-orbit of $\alpha$, that is,
\[
n_\alpha = \min \left\{ r \in \bZ_{>0} \,:\, g^r(\alpha)=\alpha
\right\}
\]
Obviously, the function $Q_1 \to \bZ_{>0}$ sending $\alpha$ to
$n_{\alpha}$ is constant on $g$-orbits.

Similarly, the invertible map $f$ induces a partition of the arrows
into $f$-orbits. Since the arrows $f(\alpha)$ and $g(\alpha)$ start
where $\alpha$ ends, the arrows of any $f$-orbit or $g$-orbit can be
arranged in a sequence whose concatenation is a cycle in $Q$.

The relations between these orbits and the triangulation are given
in the next lemma.

\begin{figure}
\begin{align*}
\begin{array}{c}
\xymatrix@=1pc{
& & & {\times} \ar@{-}@/^/[dd]^k \\
{\times} \ar@{-}@/_/[drrr]_i \ar@{-}@/^/[urrr]^j \\
& & & {\times}
}
\end{array}
& &
\begin{array}{c}
\xymatrix@=1pc{
& & & & \\
{\vdots} & {\times} \ar@{-}@/_/[drrr]^{i_0=i_n} \ar@{-}@/^/[urrr]^{i_1}
\ar@{-}@/_/[ul]_{i_2} \ar@{-}@/^/[dl]^{i_{n-1}} \\
& & & &
}
\end{array}
\end{align*}
\caption{Cycles in $Q$ arise in two ways: either from triangles of
$T$ (left) or from traversing the arcs around a puncture (right).}
\label{fig:cycles}
\end{figure}
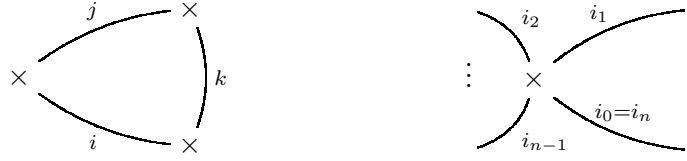

\begin{lemma} \label{l:orbits}
Let $f$, $g$ be the invertible maps corresponding to the
triangulation $T$.
\begin{enumerate}
\renewcommand{\theenumi}{\alph{enumi}}
\item
The $f$-orbits are of size $3$; they are in one-to-one
correspondence with the triangles of $T$.

\item
The $g$-orbits are of size at least $3$; they are in one-to-one
correspondence with the punctures.
\end{enumerate}
\end{lemma}
\begin{proof}
Since $f^3=id$, the $f$-orbits are of size $1$ or $3$. The first
case is impossible since always $f(\alpha) \neq \alpha$ as these
arrows start at different vertices.

Any triangle in $T$ with sides $i,j,k$ arranged in a clockwise order
as in the left drawing of Figure~\ref{fig:cycles} gives rise to a
$3$-cycle $i \to j \to k \to i$ in $Q$ which can be written as
\begin{equation} \label{e:fcycle}
\alpha \cdot f(\alpha) \cdot f^2(\alpha) .
\end{equation}
The arrows $\{\alpha, f(\alpha), f^2(\alpha)\}$ form an $f$-orbit
and any $f$-orbit is obtained in this way.

Fix a puncture, and let $i_0, i_1, \dots, i_{n-1}, i_n=i_0$ be the
sequence of arcs incident to that puncture traversed in a
counterclockwise order, as in the right drawing of
Figure~\ref{fig:cycles}. We get a cycle $i_0 \to i_1 \to \dots \to
i_{n-1} \to i_0$ in $Q$ which by construction of the map $g$ can be
written as a path
\begin{equation} \label{e:gcycle}
\beta \cdot g(\beta) \cdot \ldots \cdot g^{n-1}(\beta)
\end{equation}
whose arrows form a $g$-orbit. Moreover, any $g$-orbit is obtained
in this way.
\end{proof}

\subsection{The potential}

In~\cite{Labardini09} Labardini associates to an ideal triangulation
of a marked bordered surface a quiver with potential, using
auxiliary data consisting of a non-zero scalar for every puncture
(from a fixed field $K$). In the case of a triangulation of a marked
closed surface satisfying~\eqref{e:T3}, by the correspondence of
Lemma~\ref{l:orbits} between the punctures and the $g$-orbits on the
set of arrows $Q_1$ in the adjacency quiver, we may view the
auxiliary data as a function $c : Q_1 \to K^{\times}$ which is
constant on $g$-orbits.

Recall that a potential on a quiver $Q$ is a (possibly infinite)
linear combination of cycles in the complete path algebra
$\widehat{KQ}$ of $Q$. An explicit form of the potential associated
to the triangulation $T$ in terms of the combinatorics of its
adjacency quiver exploited in Proposition~\ref{p:quiver} is provided
by the next proposition.

\begin{prop} \label{p:potential}
Let $(Q,W)$ be the quiver with potential associated to $T$. Then the
quiver $Q$ is the adjacency quiver of $T$ described above and the
potential $W$ is given by the formula
\begin{equation} \label{e:W}
W =  \sum \alpha \cdot f(\alpha) \cdot f^2(\alpha) - \sum c_\beta
\beta \cdot g(\beta) \cdot \ldots \cdot g^{n_\beta-1}(\beta)
\end{equation}
where the first sum is taken over representatives $\alpha$ of the
$f$-orbits in $Q_1$ and the second sum is taken over representatives
$\beta$ of $g$-orbits in $Q_1$.
\end{prop}
\begin{proof}
The triangulation $T$ satisfies~\eqref{e:T3}, hence it does not
contain self-folded triangles and moreover in the formation of the
adjacency quiver no $2$-cycles had to be removed (see
Remark~\ref{rem:2cycles}). Therefore no reduction is needed which
means that the associated quiver is identical to the adjacency
quiver $Q$ of $T$ described above.

The associated potential $W$ is by definition the sum of all
$3$-cycles in $Q$ corresponding to the triangles of $T$ together
with scalar multiples of the cycles of $Q$ ``around'' each puncture
(see again Figure~\ref{fig:cycles}). By Lemma~\ref{l:orbits} and its
proof, these cycles are precisely of the forms~\eqref{e:fcycle}
and~\eqref{e:gcycle} corresponding to the $f$-orbits and $g$-orbits,
respectively.
\end{proof}

\subsection{The conditions ($\star$) and ($\diamond$)}
\label{sec:star}

In order to prove the finite-dimensionality of the Jacobian algebra
of $(Q,W)$ in full generality we need to introduce a mild condition
on the quiver $Q$ concerning the size of its $g$-orbits. This
condition is stated as follows:
\begin{equation} \tag{$\star$} \label{e:star}
\text{For any $\alpha \in Q_1$ we have $n_\alpha \geq 4$ or
$n_{f(\alpha)} \geq 4$}.
\end{equation}

Let $T$ be an ideal triangulation satisfying~\eqref{e:T3}. Then by
Lemma~\ref{l:orbits} the size of any $g$-orbit is at least $3$, that
is, $n_\alpha \geq 3$ for any $\alpha \in Q_1$. The
condition~\eqref{e:star} just says that there are not too many
$g$-orbits containing just $3$ arrows; in other words, there are not
too many punctures with just three arcs around them.

As the next lemma shows, the following condition on a triangulation
$T$ guarantees that its adjacency quiver satisfies~\eqref{e:star}:
\begin{equation} \tag{T3$\half$} \label{e:T3h}
\text{$T$ has~\eqref{e:T3} and any arc has an endpoint with at least
four arcs incident to it.}
\end{equation}

\begin{lemma} \label{l:T3h}
Let $T$ be a triangulation with property~\eqref{e:T3}. Then
condition~\eqref{e:star} is satisfied for its adjacency quiver if
and only if $T$ has property~\eqref{e:T3h}.
\end{lemma}
\begin{proof}
Assume that $T$ has property~\eqref{e:T3h}. Let $\alpha \in Q_1$ and
set $\beta=f(\alpha)$. Then $\alpha$ ends at some vertex $j$ where
$\beta$ starts at. The endpoints of the arc $j$ in the triangulation
correspond to the $g$-orbits of $\alpha$ and of $\beta$ (which may
coincide). Now the condition~\eqref{e:T3h} together with
Lemma~\ref{l:orbits} imply that at least one of these orbits
contains at least four arrows.

Conversely, assume that $T$ does not have property~\eqref{e:T3h} and
let $j$ be an arc such that both of its endpoints have only three
incident arcs. Take an arrow $\alpha \in Q_1$ ending at the vertex
$j$. Then $n_\alpha = n_{f(\alpha)} = 3$.
\end{proof}

It is much easier to verify the following property~\eqref{e:T4} of a
triangulation $T$, which obviously implies the
property~\eqref{e:T3h}:
\begin{equation} \tag{T4} \label{e:T4}
\text{at each puncture $p \in M$ there are at least four arcs of $T$
incident to $p$.}
\end{equation}

Indeed, in Section~\ref{sec:exist} we will prove that with only two
exceptions, namely the sphere with $4$ or $5$ punctures, any marked
closed surface has a triangulation with the property~\eqref{e:T4},
and that the sphere with $5$ punctures has a triangulation with
property~\eqref{e:T3h}.

However, the sphere with $4$ punctures does not have a triangulation
with property~\eqref{e:T3h}, so that an argument involving the
condition~\eqref{e:star} for the adjacency quiver of a triangulation
satisfying~\eqref{e:T3} would not be applicable. In order to deal
with this particular case, we replace the condition~\eqref{e:star}
by the condition~\eqref{e:diamond} on the $g$-orbits stated as
follows:
\begin{equation} \tag{$\diamond$} \label{e:diamond}
\text{For any $\alpha \in Q_1$ we have $n_\alpha=3$}
\end{equation}
(or equivalently, $g^3=id$ on $Q_1$) which holds for any
triangulation of a sphere with $4$ punctures having
property~\eqref{e:T3}. Under the additional assumption that the
product of the scalars associated to the punctures is not equal to
$1$, we are able to prove the finite-dimensionality in this case as
well by using similar techniques.

\section{Relations in the Jacobian algebra}
\label{sec:rels}

In this section we consider quivers with potential $(Q,W)$ of the
following form: $Q$ is any quiver with the combinatorial properties
described in Proposition~\ref{p:quiver}, and $W$ is the potential
given by the formula~\eqref{e:W} in the statement of
Proposition~\ref{p:potential}. As shown in the previous section,
this includes in particular the quivers with potential associated to
triangulations of a marked closed surface which have
property~\eqref{e:T3}.

\subsection{$\PSL_2(\bZ)$-action on the quiver}

In view of Proposition~\ref{p:quiver}\eqref{it:degs} we can make the
following definition.
\begin{defn}
For an arrow $\alpha \in Q_1$, denote by $\balpha$ the other arrow
starting at the same vertex as $\alpha$.
\end{defn}

In the next lemma we record the basic relations between the
functions $f$, $g$ and $\bar{\ }$.

\begin{lemma} \label{l:bar}
Let $\alpha \in Q_1$.
\begin{enumerate}
\renewcommand{\theenumi}{\alph{enumi}}
\item
The set $\{f^{-1}(\alpha), g^{-1}(\alpha)\}$ consists of the two
arrows that end at the vertex which $\alpha$ starts at.

\item
$\overline{f(\alpha)}=g(\alpha)$ and $\overline{g(\alpha)} =
f(\alpha)$.

\item
$gf^{-1}(\alpha) = fg^{-1}(\alpha) = \balpha$.

\item
$f^{-1}(\balpha) = g^{-1}(\alpha)$ and $g^{-1}(\balpha) =
f^{-1}(\alpha)$.

\item
$f^{-1}g(\alpha) = g^{-1}f(\alpha)$ and is equal the other arrow
ending at the same vertex as $\alpha$.
\end{enumerate}
\end{lemma}
\begin{proof}
All these claims follow from the properties of the maps $f$ and $g$
described in Proposition~\ref{p:quiver}\eqref{it:fg_def}. For
example, both arrows $f^{-1}(\alpha)$ and $g^{-1}(\alpha)$ end at
the vertex which $\alpha$ starts at. If they were identical, then
applying $f$ or $g$ would give the same arrow, namely $\alpha$, a
contradiction.

The other statements follow similarly. They are best illustrated in
the following pictures.
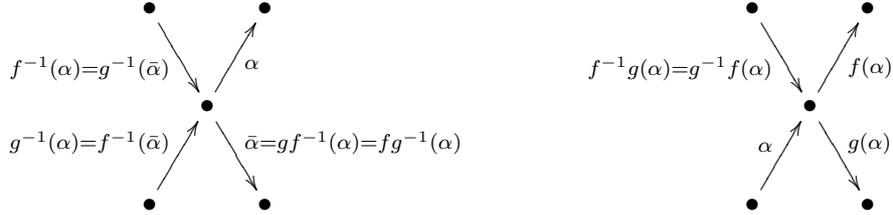
\begin{align*}
\xymatrix@=0.8pc{
{\bullet} \ar[ddr]_{f^{-1}(\alpha)=g^{-1}(\balpha)} && {\bullet} \\ \\
& {\bullet} \ar[uur]_{\alpha}
\ar[ddr]^{\balpha=gf^{-1}(\alpha)=fg^{-1}(\alpha)} \\ \\
{\bullet} \ar[uur]^{g^{-1}(\alpha)=f^{-1}(\balpha)} && {\bullet}
}
&&
\xymatrix@=0.8pc{
{\bullet} \ar[ddr]_{f^{-1}g(\alpha)=g^{-1}f(\alpha)} && {\bullet} \\ \\
& {\bullet} \ar[uur]_{f(\alpha)} \ar[ddr]^{g(\alpha)} \\ \\
{\bullet} \ar[uur]^{\alpha} && {\bullet}
}
\end{align*}
\end{proof}

\begin{prop} \label{p:PSL2}
The group $\PSL_2(\bZ)$ acts transitively on the set of arrows $Q_1$
and its subgroup consisting of all the elements acting trivially is
normal of finite index.
\end{prop}
\begin{proof}
The group $\PSL_2(\bZ)$ has a presentation by two generators $x, y$
and relations $x^2 = (xy)^3 = 1$. Its action on $Q_1$ is obtained by
letting $x, y$ act on an arrow $\alpha \in Q_1$ via
\begin{align*}
x(\alpha) = \balpha &,& y(\alpha) = g(\alpha)
\end{align*}
and noting that $(xy)(\alpha)=f(\alpha)$ by the previous lemma.

Observe that any arrow starting or ending at a vertex which $\alpha$
starts or ends at belongs to the $\PSL_2(\bZ)$-orbit of $\alpha$.
Since $Q$ is connected, this implies that the action is transitive.
\end{proof}

\begin{lemma} \label{l:fgbar}
Let $\alpha \in Q_1$, and let $i$, $j$, $k$ be the starting vertices
of the arrows $\alpha$, $f(\alpha)$ and $f^2(\alpha)$, respectively.
\begin{enumerate}
\renewcommand{\theenumi}{\alph{enumi}}
\item
The three vertices $i$, $j$, $k$ are different and the six arrows
$\alpha$, $\balpha$, $f(\alpha)$, $g(\alpha)$, $f^2(\alpha)$,
$gf(\alpha)$ are all distinct.

\item \label{it:f2}
$f^2(\alpha) = g^{n_{\balpha}-1}(\balpha)$.

\item \label{it:gf}
$gf(\alpha) = fg^{n_{\balpha}-2}(\balpha)$.
\end{enumerate}
\end{lemma}
\begin{proof}
If any two of the vertices $i,j,k$ were identical, then at least one
of the arrows $\alpha$, $f(\alpha)$ or $f^2(\alpha)$ would be a
loop, a contradiction.

Now $\alpha, \balpha$ are the two distinct arrows starting at $i$,
and similarly $f(\alpha), g(\alpha)$ are those starting at $j$ and
$f^2(\alpha), gf(\alpha)$ those starting at $k$. As $i,j,k$ are
different, we get that these six arrows are all distinct.

We illustrate the situation in the following picture
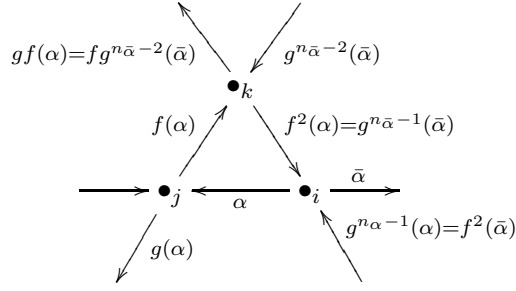
\[
\xymatrix@=0.8pc{
&& && \ar[ddl]^{g^{n_{\balpha}-2}(\balpha)} \\ \\
&& & {\bullet_k} \ar[ddr]^{f^2(\alpha)=g^{n_{\balpha}-1}(\balpha)}
\ar[uul]^{gf(\alpha)=fg^{n_{\balpha}-2}(\balpha)} \\ \\
\ar[rr] && {\bullet_j} \ar[uur]^{f(\alpha)} \ar[ddl]^{g(\alpha)}
&& {\bullet_i} \ar[ll]^{\alpha} \ar[rr]^{\balpha} && \\ \\
& &&&& \ar[uul]_{g^{n_\alpha-1}(\alpha)=f^2(\balpha)}
}
\]

Applying Lemma~\ref{l:bar}, we get $f^2(\alpha)=f^{-1}(\alpha) =
g^{-1}(\balpha) = g^{n_{\balpha}-1}(\balpha)$, hence also
\[
gf(\alpha) = (gf^{-1})f^2(\alpha) =
(fg^{-1})g^{n_{\balpha}-1}(\balpha) = fg^{n_{\balpha}-2}(\balpha) .
\]
\end{proof}

\subsection{Basic relations}

The quiver with potential $(Q,W)$ gives rise to the (completed)
Jacobian algebra $\gL = \cP(Q,W)$ which is our main object of study.
It is defined as the quotient of the completed path algebra
$\widehat{KQ}$ by the closure of the two-sided ideal generated by
the directional derivatives of $W$ with respect to all arrows.

\begin{lemma} \label{l:1f}
For any $\beta \in Q_1$ we have the following relation in $\gL$.
\[
\beta \cdot f(\beta) = c_{\bbeta} \bbeta \cdot g(\bbeta) \cdot \ldots
\cdot
g^{n_{\bbeta}-2}(\bbeta) .
\]
\end{lemma}
\begin{proof}
Since each arrow belongs to exactly one $f$-orbit and one $g$-orbit,
we see that each arrow appears exactly once in each of two sums
comprising $W$ in~\eqref{e:W}.

By computing the directional derivative of $W$ with respect to the
arrow $\alpha = f^{-1}(\beta)$ we see that
\begin{equation} \label{e:ff2}
f(\alpha) \cdot f^2(\alpha) = c_\alpha g(\alpha) \cdot g^2(\alpha)
\cdot \ldots \cdot g^{n_\alpha-1}(\alpha)
\end{equation}
and the lemma follows by noting that $\bbeta = g(\alpha)$ and hence
$n_{\bbeta} = n_\alpha$.
\end{proof}

\begin{prop}
For any $\alpha \in Q_1$ we have the following relations in $\gL$.
\begin{align}
\label{e:1ff2}
\begin{split}
\alpha \cdot f(\alpha) \cdot f^2(\alpha) &=
c_{\alpha} \alpha \cdot g(\alpha) \cdot g^2(\alpha) \cdot \ldots \cdot
g^{n_\alpha-1}(\alpha) \\ &=
c_{\balpha} \balpha \cdot g(\balpha) \cdot g^2(\balpha) \cdot \ldots \cdot
g^{n_{\balpha}-1}(\balpha) = \balpha \cdot f(\balpha) \cdot f^2(\balpha)
\end{split}
\\
\label{e:1gfg}
\alpha \cdot g(\alpha) \cdot fg(\alpha) &=
c_{f(\alpha)} \alpha \cdot f(\alpha) \cdot gf(\alpha) \cdot g^2f(\alpha) \cdot
\ldots \cdot g^{n_{f(\alpha)}-2}f(\alpha)
\\
\label{e:1fgf}
\alpha \cdot f(\alpha) \cdot gf(\alpha) &=
c_{\balpha} \balpha \cdot g(\balpha) \cdot \ldots \cdot
g^{n_{\balpha}-3}(\balpha) \cdot g^{n_{\balpha}-2}(\balpha) \cdot
fg^{n_{\balpha}-2}(\balpha)
\end{align}
\end{prop}
\begin{proof}
The first equality in~\eqref{e:1ff2} follows from~\eqref{e:ff2},
whereas the second follows from Lemma~\ref{l:1f} with $\beta=\alpha$
and Lemma~\ref{l:fgbar}(\ref{it:f2}). We get the last equality from
the first one by interchanging $\alpha$ with $\balpha$.

The relation \eqref{e:1gfg} follows from Lemma~\ref{l:1f} with
$\beta=g(\alpha)$, noting that $\bbeta = f(\alpha)$. Finally,
\eqref{e:1fgf} follows from Lemma~\ref{l:1f} with $\beta=\alpha$ and
Lemma~\ref{l:fgbar}(\ref{it:gf}).
\end{proof}

\begin{defn}
Let $i \in Q_0$ and let $\alpha$, $\balpha$ be the arrows starting
at $i$. In view of~\eqref{e:1ff2}, the two $3$-cycles
\begin{align*}
\alpha \cdot f(\alpha) \cdot f^2(\alpha) &,& \balpha \cdot
f(\balpha) \cdot f^2(\balpha)
\end{align*}
as well as the scalar multiples of the $n_{\alpha}$-cycle and
$n_{\balpha}$-cycle
\begin{align*}
c_{\alpha} \alpha \cdot g(\alpha) \cdot g^2(\alpha) \cdot \ldots
\cdot g^{n_\alpha-1}(\alpha) &,& c_{\balpha} \balpha \cdot
g(\balpha) \cdot g^2(\balpha) \cdot \ldots \cdot
g^{n_{\balpha}-1}(\balpha)
\end{align*}
starting and ending at $i$ are all equal in $\gL$. We denote their
common value by $z_i$.
\end{defn}

\subsection{Additional relations}

In this section we derive additional relations in the Jacobian
algebra under further hypotheses on the quiver. They are summarized
in the next proposition.

\begin{prop} \label{p:1gfg_zero}
Assume one of the following hypotheses:
\begin{itemize}
\item
$Q$ satisfies the condition~\eqref{e:star}; or

\item
$Q$ satisfies the condition~\eqref{e:diamond} and $\prod_{\alpha \in
\Omega} c_\alpha \neq 1$, where $\Omega$ contains one representative
from each $g$-orbit;
\end{itemize}
Then for any arrow $\alpha \in Q_1$, we have
\[
\alpha \cdot g(\alpha) \cdot fg(\alpha) = 0 \quad \text{and} \quad
\alpha \cdot f(\alpha) \cdot gf(\alpha) = 0
\]
in the completed Jacobian algebra $\gL = \cP(Q,W)$.
\end{prop}

The proof of the proposition is given by the series of lemmas below.
The case of condition~\eqref{e:star} is dealt with in
Lemma~\ref{l:1gfg_star} and Lemma~\ref{l:1gfg_zero_star}, and the
case of condition~\eqref{e:diamond} is considered in
Lemma~\ref{l:diamond} and Lemma~\ref{l:1gfg_zero_diamond}.

\begin{lemma} \label{l:1gfg_star}
Assume that $Q$ satisfies~\eqref{e:star}. Then for any arrow $\alpha
\in Q_1$, there is an arrow $\alpha' \in Q_1$ and (scalar multiples
of) paths $q$, $q'$ not both of length zero such that
\begin{align}
\label{e:1gfgq}
\alpha \cdot g(\alpha) \cdot fg(\alpha) &= \alpha \cdot f(\alpha)
\cdot gf(\alpha) \cdot q \\
\label{e:1fgfq}
\alpha \cdot f(\alpha) \cdot gf(\alpha) &= q' \cdot \alpha' \cdot
g(\alpha') \cdot fg(\alpha') .
\end{align}
\end{lemma}
\begin{proof}
The equation~\eqref{e:1gfgq} follows from~\eqref{e:1gfg}
whereas~\eqref{e:1fgfq} follows from~\eqref{e:1fgf}, taking $\alpha'
= g^{n_{\balpha}-3}(\balpha)$. The path $q$ is of length
$n_{f(\alpha)}-3$ whereas $q'$ is of length $n_{\balpha}-3$. Since
$n_{\balpha} = n_{f^2(\alpha)}$, the condition~\eqref{e:star} (for
the arrow $f(\alpha)$) implies that $n_{f(\alpha)} \geq 4$ or
$n_{\balpha} \geq 4$.
\end{proof}

\begin{lemma} \label{l:1gfg_zero_star}
Assume that $Q$ satisfies~\eqref{e:star}. Then for any arrow $\alpha
\in Q_1$, we have
\[
\alpha \cdot g(\alpha) \cdot fg(\alpha) = 0 \quad \text{and} \quad
\alpha \cdot f(\alpha) \cdot gf(\alpha) = 0
\]
in the completed Jacobian algebra $\gL$.
\end{lemma}
\begin{proof}
We show that the first expression vanishes in $\gL$. The proof for
the second is similar. Indeed, invoking~\eqref{e:1gfgq}
and~\eqref{e:1fgfq} we see that
\[
\alpha \cdot g(\alpha) \cdot fg(\alpha) = \alpha \cdot f(\alpha)
\cdot gf(\alpha) \cdot q = q' \cdot \alpha' \cdot g(\alpha') \cdot
fg(\alpha') \cdot q
\]
for some arrow $\alpha'$ and paths $q$, $q'$ not both trivial. Thus,
the path at the right hand side is strictly longer than the left
hand side and contains a subpath of the same form.

Set $\alpha_1 = \alpha'$. By repeating this process we get a
sequence $\{\alpha_m\}_{m \geq 1}$ of arrows and (scalar multiples
of) paths $q_m$, $q'_m$ whose lengths sum to at least $m$ such that
\[
\alpha \cdot g(\alpha) \cdot fg(\alpha) = q'_m \cdot \alpha_m \cdot
g(\alpha_m) \cdot fg(\alpha_m) \cdot q_m .
\]
Since $\alpha \cdot g(\alpha) \cdot fg(\alpha)$ is equal in $\gL$ to
an arbitrarily long path, we deduce that it must vanish in $\gL$.
\end{proof}

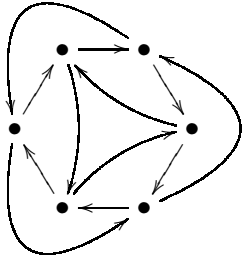
\begin{figure}
\[
\xymatrix@=0.5pc{
& {\bullet} \ar[rr] \ar@/^/[dddd]
&& {\bullet} \ar[ddr] \ar@/_3pc/[ddlll] \ar@{<-}@/^3pc/[dddd] \\ \\
{\bullet} \ar[uur] & && &
{\bullet} \ar[ddl] \ar@/^/[uulll] \ar@{<-}@/_/[ddlll] \\ \\
& {\bullet} \ar[uul]
&& {\bullet} \ar[ll] \ar@{<-}@/^3pc/[uulll]
}
\]
\caption{The quiver satisfying condition~\eqref{e:diamond}. It
arises from an ideal triangulation of a sphere with $4$ punctures
having property~\eqref{e:T3}.} \label{fig:Qdiamond}
\end{figure}

\begin{lemma} \label{l:diamond}
Assume that $Q$ satisfies~\eqref{e:diamond}. Then $Q$ is isomorphic
to the quiver shown in Figure~\ref{fig:Qdiamond}. In particular, it
has $6$ vertices, $12$ arrows and $4$ $g$-orbits. Moreover, for any
arrow $\alpha \in Q_1$, the arrows $\alpha$, $\balpha$, $f(\alpha)$
and $f(\balpha)$ belong to different $g$-orbits.
\end{lemma}
\begin{proof}
Since $g^3$ acts on $Q_1$ as the identity, the $\PSL_2(\bZ)$-action
on $Q_1$ described in Proposition~\ref{p:PSL2} induces a transitive
action of the alternating group on four elements $A_4$ which has the
presentation $\langle x,y \,:\, x^2=(xy)^3=y^3=1 \rangle$.

Therefore the number of arrows in $Q$ divides $12$. By
Lemma~\ref{l:fgbar} there are at least $6$ arrows in $Q$. The case
of $6$ arrows is impossible since then $Q$ would look like as
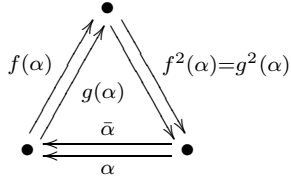
\[
\xymatrix@=1.5pc{
& {\bullet} \ar@<0.5ex>[ddr]^{f^2(\alpha)=g^2(\alpha)} \ar@<-0.5ex>[ddr]
\\ \\
{\bullet} \ar@<0.5ex>[uur]^{f(\alpha)} \ar@<-0.5ex>[uur]_{g(\alpha)} &&
{\bullet} \ar@<0.5ex>[ll]^{\alpha} \ar@<-0.5ex>[ll]_{\balpha}
}
\]
and $g^3(\alpha)=\balpha$ would imply that $Q$ does not
satisfy~\eqref{e:diamond} (as a side remark we note that such $Q$
arises from a triangulation of a once-punctured torus, and then the
$\PSL_2(\bZ)$-action induces an action of $\bZ/6\bZ$ on $Q_1$).

Therefore $Q$ has $12$ arrows and $Q_1$ is a $A_4$-torsor. The
assertions now follow. For example, the last one follows from the
fact that in the presentation of $A_4$ given above, the elements
$1,x,xy,xyx$ belong to different right cosets of the cyclic subgroup
generated by $y$.
\end{proof}

\begin{lemma} \label{l:1gfg_zero_diamond}
Assume that $Q$ satisfies~\eqref{e:diamond} and that $\prod_{\alpha
\in \Omega} c_\alpha \neq 1$, where $\Omega$ contains one
representative from each $g$-orbit. Then for any arrow $\alpha \in
Q_1$, we have
\[
\alpha \cdot g(\alpha) \cdot fg(\alpha) = 0 \quad \text{and} \quad
\alpha \cdot f(\alpha) \cdot gf(\alpha) = 0
\]
in the completed Jacobian algebra $\gL$.
\end{lemma}
\begin{proof}
By the previous lemma, there are four $g$-orbits and $\alpha$,
$\balpha$, $f(\alpha)$, $f(\balpha)$ lie in different $g$-orbits,
therefore by our assumption $c_\alpha c_{\balpha} c_{f(\alpha)}
c_{f(\balpha)} \neq 1$.

Now by repeatedly applying~\eqref{e:1gfg} and~\eqref{e:1fgf} we get
\[
\begin{split}
\alpha \cdot g(\alpha) \cdot fg(\alpha) &=
c_{f(\alpha)} \, \alpha \cdot f(\alpha) \cdot gf(\alpha) =
c_{f(\alpha)} c_{\balpha} \, \balpha \cdot g(\balpha) \cdot fg(\balpha) \\
&= c_{f(\alpha)} c_{\balpha} c_{f(\balpha)} \,
\balpha \cdot f(\balpha) \cdot gf(\balpha) =
c_{f(\alpha)} c_{\balpha} c_{f(\balpha)} c_\alpha \,
\alpha \cdot g(\alpha) \cdot fg(\alpha)
\end{split}
\]
and the result follows.
\end{proof}

\section{Jacobian algebras from ``nice'' triangulations}
\label{sec:algebra}

As in the previous section, we consider quivers with potential
$(Q,W)$ of the following form: $Q$ is any quiver with the
combinatorial properties described in Proposition~\ref{p:quiver},
and $W$ is the potential given by the formula~\eqref{e:W} in the
statement of Proposition~\ref{p:potential}.

In this section, we investigate the (completed) Jacobian algebra
$\gL = \cP(Q,W)$ under one of the following additional hypotheses:
\begin{itemize}
\item
$Q$ satisfies the condition~\eqref{e:star}; or

\item
$Q$ satisfies the condition~\eqref{e:diamond} and $\prod_{\alpha \in
\Omega} c_\alpha \neq 1$, where $\Omega$ contains one representative
from each $g$-orbit;
\end{itemize}
(so that the conclusion of Proposition~\ref{p:1gfg_zero} holds) and
show that $\gL$ is finite-dimensional, symmetric, and the potential
is not rigid. In addition we compute the Cartan matrix and the
center of $\gL$.

\subsection{Finite dimensionality}

\begin{lemma} \label{l:ffgg_zero}
For any arrow $\alpha \in Q_1$ we have in the completed Jacobian
algebra $\gL$
\[
\alpha \cdot f(\alpha) \cdot f^2(\alpha) \cdot \alpha = 0 \quad
\text{and} \quad \alpha \cdot g(\alpha) \cdot \ldots \cdot
g^{n_\alpha-1}(\alpha) \cdot \alpha = 0 .
\]
\end{lemma}
\begin{proof}
In view of~\eqref{e:1ff2}, it is enough to show that one of these
expressions vanishes, since the other is a scalar multiple of it.
Using~\eqref{e:1ff2} again, we get
\[
\alpha \cdot f(\alpha) \cdot f^2(\alpha) \cdot \alpha = \balpha
\cdot f(\balpha) \cdot f^2(\balpha) \cdot \alpha
\]
which vanishes by Proposition~\ref{p:1gfg_zero} applied to
$f(\balpha)$, noting that $\alpha=gf^2(\balpha)$.
\end{proof}

\begin{prop} \label{p:findim}
The algebra $\gL$ is finite-dimensional. It has a basis consisting
of the paths
\[
\left\{e_i\right\}_{i \in Q_0} \cup \left\{\alpha \cdot g(\alpha)
\cdot \ldots \cdot g^r(\alpha)\right\}_{\alpha \in Q_1,\, 0 \leq r <
n_\alpha-1} \cup \{z_i\}_{i \in Q_0} .
\]
\end{prop}
\begin{proof}
To show the finite-dimensionality of $\gL$, it is enough to show
that the image of any sufficiently long path in $\widehat{KQ}$
vanishes.

Indeed, by Proposition~\ref{p:quiver}(\ref{it:fg_def}) such a path
can be written as $\alpha_0 \cdot \alpha_1 \cdot \ldots \cdot
\alpha_N$ where for every $0 \leq j < N$ we have $\alpha_{j+1} =
f(\alpha_j)$ or $\alpha_{j+1}=g(\alpha_j)$.

Now Proposition~\ref{p:1gfg_zero} and Lemma~\ref{l:ffgg_zero} tell
us that the only paths whose image in $\gL$ is possibly non-zero are
of the form $\alpha \cdot g(\alpha) \cdot \ldots \cdot g^r(\alpha)$
for some $0 \leq r \leq n_\alpha-1$ or paths of the form $\alpha
\cdot f(\alpha)$ or $\alpha \cdot f(\alpha) \cdot f^2(\alpha)$. As
there are only finitely many such paths, this shows the
finite-dimensionality of $\gL$.

By Lemma~\ref{l:1f} and~\eqref{e:1ff2}, any path $\beta \cdot
f(\beta)$ or $\beta \cdot f(\beta) \cdot f^2(\beta)$ can be
expressed as (scalar multiple) of a path of the form $\alpha \cdot
g(\alpha) \cdot \ldots \cdot g^r(\alpha)$ for suitable $\alpha \in
Q_1$ and $r \geq 0$. Therefore the algebra $\gL$ is spanned by the
trivial paths $e_i$ for each vertex $i \in Q_0$ together with the
paths $\alpha \cdot g(\alpha) \cdot \ldots \cdot g^r(\alpha)$ for
$\alpha \in Q_1$ and $0 \leq r \leq n_\alpha-1$. The only relations
among these paths are the commutativity relations in~\eqref{e:1ff2},
hence when forming a basis for $\gL$ we may take only those paths
with $r<n_\alpha-1$ and add the cycle $z_i$ for each $i \in Q_0$.
\end{proof}

\begin{remark}
The algebra $\gL$ can therefore be written as a quiver with
relations $\gL \simeq KQ/I$. The description of the ideal $I
\subseteq KQ$ depends on our hypothesis on $Q$; if $Q$ satisfies the
condition~\eqref{e:diamond}, then
\[
I = \left\langle \alpha \cdot f(\alpha) - c_{\balpha} \balpha \cdot
g(\balpha) \cdot \ldots \cdot g^{n_{\balpha}-2}(\balpha) \,:\,
\alpha \in Q_1 \right\rangle ,
\]
whereas if $Q$ satisfies the condition~\eqref{e:star}, then
\[
I = \left\langle \alpha \cdot f(\alpha) - c_{\balpha} \balpha \cdot
g(\balpha) \cdot \ldots \cdot g^{n_{\balpha}-2}(\balpha),\, \beta
\cdot f(\beta) \cdot gf(\beta) \,:\, \alpha \in Q_1,\, \beta \in
\Theta \right\rangle
\]
where $\Theta \subseteq Q_1$ is a set of representatives of
$h$-orbits for the (invertible) map $h \colon Q_1 \to Q_1$ defined
by $h(\beta)=g^{-3}(\bbeta)$.
\end{remark}

\subsection{Non-rigidity}

\begin{prop} \label{p:nonrigid}
The potential $W$ is not rigid.
\end{prop}
\begin{proof}
By~\cite[\S8]{DWZ08}, in order to show that $W$ is not rigid, one
has to find a potential on $Q$ which is not cyclically equivalent to
an element in the Jacobian ideal of $W$.

Indeed, consider a potential $W'$ in $\widehat{KQ}$ which is a
$3$-cycle $W' = \alpha \cdot f(\alpha) \cdot f^2(\alpha)$ for some
arrow $\alpha \in Q_1$ starting at some vertex $i$. The image of
$W'$ in $\gL$ is $z_i \neq 0$, hence it does not belong to the
Jacobian ideal of $W$. Since this holds for any such $3$-cycle, we
deduce that $W'$ is not cyclically equivalent to an element of the
Jacobian ideal of $W$, hence $W$ is not rigid.
\end{proof}

\subsection{Symmetry}
For a finite-dimensional algebra $\gL$ the space $D\gL =
\Hom_K(\gL,K)$ of $K$-linear functionals on $\gL$ is a
$\gL$-$\gL$-bimodule via
\begin{align*}
(\vphi \gl)(x) = \vphi(\gl x) && (\gl \vphi)(x) = \vphi(x \gl)
\end{align*}
for $\vphi \in D\gL$ and $\gl, x \in \gL$. The algebra $\gL$ is
called \emph{symmetric} if $D\gL$ and $\gL$ are isomorphic as
$\gL$-$\gL$-bimodules. For an element $\gl \in \gL$, define a dual
element $\gl^{\vee} \in D\gL$ by
\[
\gl^{\vee}(x) =
\begin{cases}
a & \text{if $x = a\gl$ for some $a \in K$} \\
0 & \text{otherwise}
\end{cases}
\]
so that $(c \lambda)^\vee = c^{-1} \lambda^\vee$ for any $c \in
K^{\times}$.

Let $\gL = \cP(Q,W)$. There is a duality between paths which can be
extended to a $K$-linear isomorphism $\Phi : D\gL \xrightarrow{\sim}
\gL$ defined in the following way. Any non-zero path in $\gL$ has
the form $p = \alpha \cdot g(\alpha) \cdot \ldots \cdot
g^{r-1}(\alpha)$ for some $\alpha \in Q_1$ and $0 \leq r \leq
n_\alpha$ (here, $r=0$ means the path of length $0$ corresponding to
the starting vertex of $\alpha$). The path $p$ can be completed to a
cycle $p \cdot q$ along a $g$-orbit and we define $\Phi(p^\vee)$ to
be the multiple of $q$ by the scalar corresponding to that
$g$-orbit. More precisely,
\begin{equation} \label{e:Phi}
\Phi((\alpha \cdot g(\alpha) \cdot \ldots \cdot
g^{r-1}(\alpha))^\vee) = c_\alpha g^r(\alpha) \cdot \ldots \cdot
g^{n_\alpha-1}(\alpha)
\end{equation}
which is well-defined by the identity~\eqref{e:1ff2}. In particular,
$\Phi(e_i^\vee) = z_i$ and $\Phi(z_i^\vee) = e_i$ for any $i \in
Q_0$. Since $D\gL$ has a basis
\[
\left\{e_i^\vee \right\}_{i \in Q_0} \cup \left\{ \left(\alpha \cdot
g(\alpha) \cdot \ldots \cdot g^{r-1}(\alpha) \right)^{\vee}
\right\}_{\alpha \in Q_1,\, 1 \leq r \leq n_\alpha-1} \cup
\left\{z_i^\vee \right\}_{i \in Q_0}
\]
$\Phi$ can be extended by linearity to a $K$-linear isomorphism
$\Phi : D\gL \to \gL$.

The next lemma shows that $\Phi$ has a similar completion property
with respect to $f$-orbits as well.

\begin{lemma} \label{l:Phi_f}
Let $\alpha \in Q_1$. Then
\begin{align*}
\Phi(\alpha^\vee) = f(\alpha) \cdot f^2(\alpha) && \Phi((\alpha
\cdot f(\alpha))^\vee) = f^2(\alpha) .
\end{align*}
\end{lemma}
\begin{proof}
This follows from~\eqref{e:ff2}, Lemma~\ref{l:1f} and
Lemma~\ref{l:fgbar}(\ref{it:f2}).
\end{proof}

\begin{lemma} \label{l:ziZ}
Let $i \in Q_0$, $\alpha \in Q_1$. Then $\alpha \cdot z_i = 0$ and
$z_i \cdot \alpha = 0$.
\end{lemma}
\begin{proof}
We show only that $z_i \cdot \alpha = 0$. The proof of the other
claim is similar. Since $z_i$ is a cycle starting and ending at $i$,
it is enough to consider an arrow $\alpha$ starting at $i$. But then
we can write $z_i = \alpha \cdot f(\alpha) \cdot f^2(\alpha)$ and
the result follows from Lemma~\ref{l:ffgg_zero}.
\end{proof}

\begin{prop} \label{p:symmetric}
The Jacobian algebra $\gL$ of $(Q,W)$ is symmetric.
\end{prop}
\begin{proof}
We show that the isomorphism of $K$-vector spaces $\Phi : D\gL \to
\gL$ is an isomorphism of $\gL$-$\gL$-bimodules. In other words, we
need to verify that for any path $p$ in $\gL$ and any $i \in Q_0$,
$\beta \in Q_1$ we have
\begin{align}
\label{e:Phi_ei}
\Phi(p^\vee \cdot e_i) &= \Phi(p^\vee) \cdot e_i &&
\Phi(e_i \cdot p^\vee) = e_i \cdot \Phi(p^\vee) \\
\label{e:Phi_beta}
\Phi(p^\vee \cdot \beta) &= \Phi(p^\vee) \cdot \beta &&
\Phi(\beta \cdot p^\vee) = \beta \cdot \Phi(p^\vee) .
\end{align}

If $p$ starts at $i$ and ends at $j$, then $\Phi(p)$ is a multiple
of a path starting at $j$ and ending at $i$. This
shows~\eqref{e:Phi_ei}. For~\eqref{e:Phi_beta}, we start by noting
that $p^\vee \cdot \beta = 0$ if $p$ cannot be written as a linear
combination of paths starting at $\beta$ and $p^\vee \cdot \beta =
q^\vee$ if $p$ can be written uniquely as $p = \beta q$, and
similarly for $\beta \cdot p^\vee$.

Let $p = e_i$ for $i \in Q_0$ and let $\beta \in Q_1$. Then
$e_i^\vee \cdot \beta = 0$, $\beta \cdot e_i^\vee = 0$
and~\eqref{e:Phi_beta} follows from Lemma~\ref{l:ziZ}.

Let $p = z_i$ for some $i \in Q_0$ and let $\alpha$, $\balpha$ be
the arrows starting at $i$. Then
\begin{align*}
z_i^\vee \cdot \beta &=
\begin{cases}
(c_\alpha g(\alpha) \cdot \ldots \cdot g^{n_\alpha-1}(\alpha))^\vee
& \text{if $\beta = \alpha$,} \\
(c_{\balpha} g(\balpha) \cdot \ldots \cdot
g^{n_{\balpha}-1}(\balpha))^\vee
& \text{if $\beta = \balpha$,} \\
0 & \text{otherwise}
\end{cases}
\\
\beta \cdot z_i^\vee &=
\begin{cases}
(c_\alpha \alpha \cdot g(\alpha) \cdot \ldots \cdot
g^{n_\alpha-2}(\alpha))^\vee
& \text{if $\beta = g^{n_\alpha-1}(\alpha)$,} \\
(c_{\balpha} \balpha \cdot g(\balpha) \cdot \ldots \cdot
g^{n_{\balpha}-2}(\balpha))^\vee
& \text{if $\beta = g^{n_{\balpha}-1}(\balpha)$,} \\
0 & \text{otherwise}
\end{cases}
\end{align*}
and~\eqref{e:Phi_beta} follows from~\eqref{e:Phi}.

Let $p = \alpha \cdot g(\alpha) \cdot \ldots \cdot g^{r-1}(\alpha)$
for some $\alpha \in Q_1$. If $1 \leq r < n_\alpha-1$, then
\begin{align*}
p^\vee \cdot \beta &=
\begin{cases}
(g(\alpha) \cdot \ldots \cdot g^{r-1}(\alpha))^\vee
& \text{if $\beta=\alpha$,} \\
0 & \text{otherwise}
\end{cases}
\\
\beta \cdot p^\vee &=
\begin{cases}
(\alpha \cdot g(\alpha) \cdot \ldots \cdot g^{r-2}(\alpha))^\vee
& \text{if $\beta=g^{r-1}(\alpha)$,} \\
0 & \text{otherwise}
\end{cases}
\end{align*}
and~\eqref{e:Phi_beta} follows from~\eqref{e:Phi}. Finally, if
$r=n_\alpha-1$ then $p=c_\alpha^{-1} \balpha \cdot f(\balpha)$ by
Lemma~\ref{l:1f} so that by Lemma~\ref{l:Phi_f},
$\Phi(p^\vee) = c_\alpha f^2(\balpha) = c_\alpha
g^{n_\alpha-1}(\alpha)$ and
\begin{align*}
p^\vee \cdot \beta &=
\begin{cases}
(g(\alpha) \cdot \ldots \cdot g^{n_\alpha-2}(\alpha))^\vee
& \text{if $\beta=\alpha$,} \\
c_\alpha f(\balpha)^\vee
& \text{if $\beta=\balpha$,} \\
0 & \text{otherwise}
\end{cases}
\\
\beta \cdot p^\vee &=
\begin{cases}
(\alpha \cdot g(\alpha) \cdot \ldots \cdot
g^{n_\alpha-3}(\alpha))^\vee
& \text{if $\beta=g^{n-2}(\alpha)$,} \\
c_\alpha \balpha^\vee
& \text{if $\beta=f(\balpha)$,} \\
0 & \text{otherwise}
\end{cases}
\end{align*}
thus~\eqref{e:Phi_beta} follows from~\eqref{e:Phi},
Lemma~\ref{l:fgbar}(\ref{it:f2}) and Lemma~\ref{l:Phi_f}.
\end{proof}

\subsection{The Cartan matrix}
Recall that the Cartan matrix of $\gL$ is a $|Q_0| \times |Q_0|$
matrix whose $(i,j)$-entry is the dimension of the space of paths in
$\gL$ starting at the vertex $i$ and ending at $j$.

Any puncture $p \in M$ defines a column vector $v_p \in \bZ^{Q_0}$
in the following way. Let $i_0,i_1,\dots,i_{n-1},i_n=i_0$ be the
sequence of arcs incident to $p$ traversed in a counterclockwise
order, so that $i_0 \to i_1 \to \dots \to i_{n-1} \to i_0$ is a
cycle whose arrows from a $g$-orbit. For any arc $i$ set $v_p(i)$ to
be the number of times $i$ appears in the sequence $(i_0, i_1,
\dots, i_{n-1})$. Set also $n_p=n$, or equivalently $n_p = \sum_{i
\in Q_0} v_p(i)$.

\begin{prop} \label{p:Cartan}
The Cartan matrix of $\gL$ is given by the formula
\[
C_{\gL} = \sum_{p \in M} v_p \cdot v^T_p
\]
or equivalently, $(C_\gL)_{ij} = \sum_{p \in M} v_p(i) v_p(j)$.
Moreover, $(C_\gL)_{ii} \in \{2,4\}$ and $(C_\gL)_{ij} \in
\{0,1,2,4\}$ for any $i,j \in Q_0$. In particular, $\dim_K \gL =
\sum_{p \in M} n_p^2$.
\end{prop}
\begin{proof}
Consider two different vertices $i$ and $j$. Every non-zero path
from $i$ to $j$ is of the form $\alpha \cdot g(\alpha) \cdot \ldots
\cdot g^r(\alpha)$ for suitable $\alpha \in Q_1$ and $r \geq 0$, and
all these paths are linearly independent. Hence the entry
$(C_\gL)_{ij}$ is just the number of such paths.

Any such path corresponds to traversal of the arcs around a puncture
starting at the arc $i$ and ending at $j$ going at a
counterclockwise direction without completing a full round. For a
given puncture $p$, the number of such traversals is therefore
$v_p(i) v_p(j)$, hence the number of all such paths is $\sum_{p \in
M} v_p(i) v_p(j)$.

If $i=j$ then in this way we have not counted the trivial path
$e_i$, but on the other hand we counted the cycle $z_i$ twice in
view of the commutativity relations~\eqref{e:1ff2}, so the formula
$(C_\gL)_{ii} = \sum_{p \in M} v_p(i) v_p(i)$ still holds.

The remaining assertions on the entries $(C_\gL)_{ij}$ follow from
the fact that for any $i \in Q_0$, $v_p(i) \geq 0$ for $p \in M$ and
$\sum_{p \in M} v_p(i)=2$.
\end{proof}

\begin{cor} \label{c:Cartan}
We have $\rank C_\gL \leq |M|$ and $\det C_\gL = 0$.
\end{cor}
\begin{proof}
The rank of each of the $|M|$ summands $v_p v_p^T$ of $C_{\gL}$ is
$1$, hence the first claim. The second claim follows now from the
fact that always $|M| < |Q_0|$. Indeed, the number of arcs in a
triangulation of a closed surface with genus $g$ and $P$ punctures
is $6g-6+3P$ which always exceeds $P$.
\end{proof}

\begin{remark}
The vanishing of $\det C_\gL$ comes in stark contrast to the
situation for the Jacobian algebras arising from triangulations of
bordered surfaces without punctures. Indeed, these algebras are
gentle~\cite{ABCP10} and their Cartan determinants are always powers
of $2$ by~\cite{Holm05}.
\end{remark}

\subsection{The center}

\begin{prop} \label{p:center}
The center $Z(\gL)$ of $\gL$ is isomorphic to the truncated
polynomial algebra $K[\{x_i\}_{i \in Q_0}]/(\{x_i x_j\}_{i,j \in
Q_0})$.
\end{prop}
\begin{proof}
We show that a basis of $Z(\gL)$ is given by $1$ together with the
cycles $z_i$ for each $i \in Q_0$. The relation $z_i \cdot z_j = 0$
would follow from Lemma~\ref{l:ziZ}.

Let $z \in Z(\gL)$. Since $z$ commutes with the idempotents $e_i$,
it must be a sum of cycles. Let us describe the non-zero cycles
starting at a given vertex $i \in Q_0$. Obviously, $e_i$ and $z_i$
are such cycles. Let $\alpha$ be an arrow starting at $i$. If
$\balpha$ and $\alpha$ are not in the same $g$-orbit, then these are
all such cycles, otherwise write $\balpha=g^r(\alpha)$ and then
\begin{align*}
w_i=\alpha \cdot g(\alpha) \cdot \ldots \cdot g^{r-1}(\alpha) &,&
w_i'=g^r(\alpha) \cdot g^{r+1}(\alpha) \cdot \ldots \cdot
g^{n_\alpha-1}(\alpha)
\end{align*}
are also non-zero cycles starting at $i$ as in the following
picture,
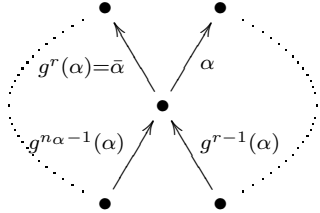
\[
\xymatrix@=0.8pc{
{\bullet} \ar@{.}@/_3pc/[dddd] & & {\bullet} \ar@{.}@/^3pc/[dddd] \\ \\
& {\bullet} \ar[uur]_{\alpha} \ar[uul]^{g^r(\alpha)=\balpha} \\ \\
{\bullet} \ar[uur]^{g^{n_\alpha-1}(\alpha)} & &
{\bullet} \ar[uul]_{g^{r-1}(\alpha)}
}
\]
and together with $e_i$ and $z_i$ they form all such cycles.

Assume that $\alpha$ and $\balpha$ are in the same $g$-orbit. We
want to show that in $z$ the coefficients of the cycles $w_i$ and
$w'_i$ must vanish. Indeed, write
\[
z = \lambda_i e_i + \mu_i z_i + \rho_i w_i + \rho'_i w'_i + \ldots
\]
for some scalars $\lambda_i, \mu_i, \rho_i, \rho'_i$ where we ignore
all cycles not starting at $i$.

Since there are no $2$-cycles in $Q$, we have $3 \leq r \leq
n_\alpha-3$ and $w_i' \cdot \balpha = 0$ by
Proposition~\ref{p:1gfg_zero}. Thus, if $\rho_i \neq 0$, then by
Lemma~\ref{l:ziZ}
\[
z \cdot \balpha = \lambda_i \balpha + \rho_i w_i \cdot \balpha =
\lambda_i \balpha + \rho_i \alpha \cdot \ldots \cdot g^r(\alpha)
\]
whereas $\balpha \cdot z$ is a sum of paths all starting at
$\balpha$. Since $\alpha \cdot \ldots \cdot g^r(\alpha)$ cannot be
written as a sum of paths starting at $\balpha$, we get that $z
\cdot \balpha \neq \balpha \cdot z$, a contradiction. We deduce that
$\rho_i = 0$. A similar argument with multiplication by $\alpha$
shows that $\rho'_i = 0$ as well.

Finally note that all the coefficients $\lambda_i$ must be equal,
since $Q$ is connected, whereas there is no restriction on the
coefficients $\mu_i$ in view of Lemma~\ref{l:ziZ}.
\end{proof}

\section{Existence of ``nice'' triangulations}
\label{sec:exist}

\begin{prop}
Let $(S,M)$ be a marked closed surface. Then:
\begin{enumerate}
\renewcommand{\theenumi}{\alph{enumi}}
\item
If $(S,M)$ is not a sphere with $4$ or $5$ punctures, it has a
triangulation satisfying~\eqref{e:T4}.

\item
If $(S,M)$ is a sphere with $5$ punctures, it has a triangulation
satisfying~\eqref{e:T3h}, but no triangulation
satisfying~\eqref{e:T4}.

\item
If $(S,M)$ is a sphere with $4$ punctures, it has a triangulation
satisfying~\eqref{e:T3}, but no triangulation
satisfying~\eqref{e:T3h}.
\end{enumerate}
\end{prop}

The proof is by induction on the number of punctures, and follows by
combining the statements of the next lemmas.

\begin{lemma}
Let $(S,M)$ be a marked closed surface and $(S,M')$ the marked
closed surface obtained from $(S,M)$ by adding one more puncture.
\begin{enumerate}
\renewcommand{\theenumi}{\alph{enumi}}
\item
If $(S,M)$ has a triangulation satisfying~\eqref{e:T3}, then so does
$(S,M')$.

\item
If $(S,M)$ has a triangulation satisfying~\eqref{e:T3h}, then so
does $(S,M')$.

\item
If $(S,M)$ has a triangulation satisfying~\eqref{e:T4}, then so does
$(S,M')$.
\end{enumerate}
\end{lemma}
\begin{proof}
Let $T$ be a triangulation of $(S,M)$ without self-folded triangles.
We may place the additional puncture $p$ of $M'$ on an arc of $T$
and obtain a triangulation $T'$ of $(S,M')$ by adding four arcs
incident to $p$ as in the right picture below:
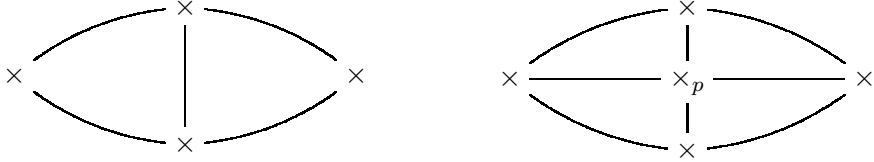
\begin{align*}
\xymatrix@=1pc{
& & & {\times} \ar@{-}[dd] \\
{\times} \ar@{-}@/_/[drrr] \ar@{-}@/^/[urrr] &&& &&&
{\times} \ar@{-}@/^/[dlll] \ar@{-}@/_/[ulll] \\
& & & {\times}
}
&&
\xymatrix@=1pc{
& & & {\times} \ar@{-}[d] \\
{\times} \ar@{-}@/_/[drrr] \ar@{-}@/^/[urrr] &&&
{\times_p} \ar@{-}[lll] \ar@{-}[rrr]
&&&
{\times} \ar@{-}@/^/[dlll] \ar@{-}@/_/[ulll] \\
& & & {\times} \ar@{-}[u]
}
\end{align*}
In $T'$ there are $4$ arcs incident to $p$ and the number of arcs
incident to each other puncture has not decreased. The lemma thus
follows.
\end{proof}

\begin{lemma}
Any triangulation of a once-punctured closed surface of genus $g
\geq 1$ has property~\eqref{e:T4}.
\end{lemma}
\begin{proof}
When counting the arcs incident to the puncture, each arc of the
triangulation is counted twice. Since there are $6g-3$ arcs in the
triangulation, the puncture has $12g-6$ arcs incident to it.
\end{proof}

\begin{lemma}
\begin{enumerate}
\renewcommand{\theenumi}{\alph{enumi}}
\item
A sphere with $6$ punctures has a triangulation
satisfying~\eqref{e:T4}.

\item
A sphere with $5$ punctures has a triangulation
satisfying~\eqref{e:T3h}, but no triangulation
satisfying~\eqref{e:T4}.

\item
A sphere with $4$ punctures has a triangulation
satisfying~\eqref{e:T3}, but no triangulation
satisfying~\eqref{e:T3h}.
\end{enumerate}
\end{lemma}
\begin{proof}
Figure~\ref{fig:tri456} presents triangulations of spheres with $4$,
$5$ and $6$ punctures with the required properties. Note that they
can be viewed as the faces of a tetrahedron, triangular bipyramid
and an octahedron, respectively.

No triangulation of a sphere with $4$ or $5$ punctures can
satisfy~\eqref{e:T4}, since the number of arcs ($6$ and $9$,
respectively) is less than twice the number of punctures. Moreover,
a triangulation of a sphere with $4$ punctures which
satisfies~\eqref{e:T3} cannot satisfy~\eqref{e:T3h}, since at all
the punctures there are exactly three incident arcs.
\end{proof}

\begin{figure}
\begin{align*}
\begin{array}{c}
\xymatrix@=1.5pc{
& {\times} \ar@{-}@/^4pc/[dd] \ar@{-}@/_/[dr] \\
{\times} \ar@{-}@/^/[ur] \ar@{-}[rr] &&
{\times} \ar@{-}@/_/[dl] \\
& {\times} \ar@{-}@/^/[ul]
}
\end{array}
&&
\begin{array}{c}
\xymatrix@=0.3pc{
&& {\times} \ar@{-}[ddrr] \ar@{-}@/_4pc/[dddddl] \ar@{-}@/^4pc/[dddddr] \\ \\
{\times} \ar@{-}[uurr] && &&
{\times} \ar@{-}[dddl] \ar@{-}@/_/[llll] \ar@{-}@/_/[dddlll] \\ \\ \\
& {\times} \ar@{-}[uuul] && {\times} \ar@{-}@/^/[ll]
}
\end{array}
&&
\begin{array}{c}
\xymatrix@=0.5pc{
\\
& {\times} \ar@{-}[rr] \ar@{-}@/^/[dddd]
&& {\times} \ar@{-}[ddr] \ar@{-}@/_3pc/[ddlll] \ar@{-}@/^3pc/[dddd] \\ \\
{\times} \ar@{-}[uur] & && &
{\times} \ar@{-}[ddl] \ar@{-}@/^/[uulll] \ar@{-}@/_/[ddlll] \\ \\
& {\times} \ar@{-}[uul]
&& {\times} \ar@{-}[ll] \ar@{-}@/^3pc/[uulll]
\\
&
}
\end{array}
\end{align*}
\caption{Triangulations of spheres with $4$, $5$ and $6$ punctures.}
\label{fig:tri456}
\end{figure}
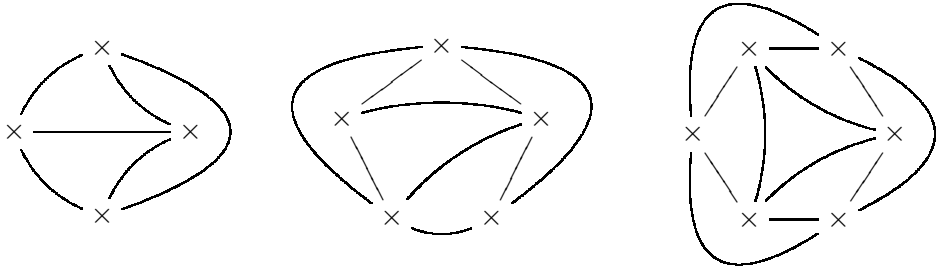

\bibliographystyle{amsplain}
\bibliography{findim}

\end{document}